\documentclass[envcount]{llncs}

\usepackage{amsmath}
\usepackage{amssymb}
\usepackage{paralist}
\usepackage{upref}
\usepackage{appendix}
\usepackage{graphicx}					
\usepackage{todonotes}					
\usepackage{tikz}                                     
\usepackage[shortlabels]{enumitem} 	
\usepackage{subcaption}				
\usepackage{hyperref}                          	
\hypersetup{
    colorlinks=true,
    citecolor = blue,
    linkcolor=blue,
    filecolor=magenta,
    urlcolor=cyan,
    bookmarks=true,
}

\newcommand{\appref}[1]{\hyperref[#1]{Appendix~\ref{#1}}}
\newcommand\emitem[1]{\item{\itshape #1}\\}		
\newcommand \e {\hfill {\tiny $\blacksquare$}}		

\let\doendproof\endproof
\renewcommand\endproof{\hfill\qed\doendproof}

\setlist[1]{itemsep=0.5em}		

\title
{
Linearly $\chi$-Bounding $(P_6,C_4)$-Free Graphs\thanks{An extended abstract of this paper appeared in the proceedings of WG 2017.}
}

\author{Serge Gaspers\inst{1,2}\and Shenwei Huang\inst{3}\thanks{Corresponding author (shenweihuang@nankai.edu.cn). }}
\institute{
School of Computer Science and Engineering\\
UNSW Sydney, Sydney 2052, Australia\\
\and Decision Sciences, Data61, CSIRO, Sydney 2052, Australia.\\
\and College of Computer Science, Nankai University, Tianjin 300071, China. \\
}

\date{}

\begin{document}

\maketitle

\begin{abstract}
Given two graphs $H_1$ and $H_2$, a graph $G$ is $(H_1,H_2)$-free if it contains no induced subgraph isomorphic to $H_1$ or $H_2$. 
Let $P_t$ and $C_s$ be the path on $t$ vertices and the cycle on $s$ vertices, respectively.
In this paper we show that for any $(P_6,C_4)$-free graph $G$
it holds that $\chi(G)\le \frac{3}{2}\omega(G)$, where $\chi(G)$ and $\omega(G)$ are the chromatic number
and clique number of $G$, respectively. 
Our bound is attained by several graphs, for instance,
the five-cycle,  the Petersen graph, the Petersen graph with an additional universal vertex, and
all $4$-critical $(P_6,C_4)$-free graphs other than $K_4$ (see \cite{HH17}).
The new result unifies previously known results
on the existence of linear $\chi$-binding functions for several graph classes.
Our proof is based on a novel structure theorem on $(P_6,C_4)$-free graphs
that do not contain clique cutsets. Using this structure theorem we also design a polynomial time $3/2$-approximation 
algorithm for coloring  $(P_6,C_4)$-free graphs. 
Our algorithm computes a coloring with $\frac{3}{2}\omega(G)$ colors for any $(P_6,C_4)$-free graph $G$ in $O(n^2m)$ time.
\end{abstract}

\section{Introduction}

All graphs in this paper are finite and simple.
We say that a graph $G$ {\em contains} a graph $H$ if $H$ is
isomorphic to an induced subgraph of $G$.  A graph $G$ is
{\em $H$-free} if it does not contain $H$. 
For a family of graphs $\mathcal{H}$,
$G$ is {\em $\mathcal{H}$-free} if $G$ is $H$-free for every $H\in \mathcal{H}$.
In case that $\mathcal{H}$ consists of two graphs, we  write
$(H_1,H_2)$-free instead of $\{H_1,H_2\}$-free.
As usual, let $P_t$ and $C_s$ denote
the path on $t$ vertices and the cycle on $s$ vertices, respectively. The complete
graph on $n$ vertices is denoted by $K_n$.
For two graphs $G$ and $H$, we use $G+H$ to denote the \emph{disjoint union} of $G$ and $H$.
The \emph{join} of $G$ and $H$, denoted by $G\vee H$, is the graph obtained by taking the disjoint union
of $G$ and $H$ and adding an edge between every vertex in $G$ and every vertex in $H$.
For a positive integer $r$, we use $rG$ to denote the disjoint union of $r$ copies of $G$.
The \emph{complement} of $G$ is denoted by $\overline{G}$.
The \emph{girth} of $G$ is the length of the shortest cycle in $G$.
A \emph{$q$-coloring} of a graph $G$ is a function $\phi:V(G)\longrightarrow \{ 1, \ldots ,q\}$ such that
$\phi(u)\neq \phi(v)$ whenever $u$ and $v$ are adjacent in $G$.
The \emph{chromatic number} of a graph $G$, denoted by
$\chi (G)$, is the minimum number $q$ for which there exists a $q$-coloring of $G$.
The \emph{clique number} of $G$, denoted by $\omega(G)$, is the size of a largest clique in $G$.
Obviously, $\chi(G)\ge \omega(G)$ for any graph $G$.

A family $\mathcal{G}$ of graphs is said to be \emph{$\chi$-bounded} if 
there exists a function $f$ such that for every graph
$G\in \mathcal{G}$ and every induced subgraph $H$ of $G$ it holds that
$\chi(H)\le f(\omega(H))$. The function $f$ is called a \emph{$\chi$-binding} function
for $\mathcal{G}$. The class of perfect graphs (a graph $G$ is \emph{perfect} if for every induced subgraph $H$ of $G$ 
it holds that $\chi(H)=\omega(H)$), for instance, is a  $\chi$-bounded family with $\chi$-binding
function $f(x)=x$. 
Therefore, $\chi$-boundedness is a generalization of perfection.
The notion of $\chi$-bounded families was introduced by Gy{\'a}rf{\'a}s \cite{Gy87} who posed the following two 
meta problems:

\begin{itemize}
\item [$\bullet$] Does there exist a $\chi$-binding function $f$ for a given family $\mathcal{G}$ of graphs?
\item [$\bullet$] Does there exist a \emph{linear} $\chi$-binding function $f$ for $\mathcal{G}$?
 \end{itemize} 

The two problems have received considerable attention for hereditary classes.
Hereditary classes are exactly those classes that can be characterized by \emph{forbidden induced subgraphs}.
What choices of forbidden induced subgraphs guarantee that a family of graphs is $\chi$-bounded? 
Since there are graphs with arbitrarily large chromatic number
and girth \cite{Er59}, at least one forbidden subgraph has to be acyclic.
Gy{\'a}rf{\'a}s \cite{Gy73} conjectured that this necessary condition is also a
sufficient condition for a hereditary class to be $\chi$-bounded.

\begin{conjecture}[Gy{\'a}rf{\'a}s \cite{Gy73}]
For every forest $T$, the class of $T$-free graphs is $\chi$-bounded.
\end{conjecture}

Gy{\'a}rf{\'a}s \cite{Gy87} proved the conjecture for $T=P_t$: every $P_t$-free graph
$G$ has $\chi(G)\le (t-1)^{\omega(G)-1}$.  Note that this $\chi$-binding function is exponential
in $\omega(G)$. Therefore, it is natural to ask whether there exists a linear $\chi$-binding function
for $P_t$-free graphs. Unfortunately, unless $t\le 4$ in which case every $P_t$-free graph is perfect
and hence has $\chi(G)=\omega(G)$, no linear $\chi$-binding function exists for $P_t$-free graphs when $t\ge 5$ 
\cite{FGMT95}. In fact, as observed in \cite{RS04}, the class of $H$-free graphs admits a linear $\chi$-binding function
if and only if $H$ is contained in a $P_4$.

However, if an additional graph is forbidden, then the class could become linearly $\chi$-bounded again.
Choudum, Karthick and Shalu \cite{CKS07} derived a linear $\chi$-binding function for
$(P_6,P_4\vee P_1)$-free graphs, $(P_5,P_4\vee P_1)$-free graphs and $(P_5,C_4\vee P_1)$-free graphs. 
In the same paper, they also obtained the optimal $\chi$-binding function $f(x)=\lceil \frac{5}{4}x \rceil$ for $(P_5,C_4)$-free graphs,
improving a result in \cite{FGMT95}. Later on, the same set of authors \cite{CKS08} obtained linear $\chi$-binding functions
for certain subclasses of $3P_1$-free graphs (thus subclasses of $P_5$-free graphs). In particular, they showed that
the class of $(3P_1, K_4+P_1)$-free graphs has a linear $\chi$-binding function $f(x)=2x$.
Henning, L\"{o}wenstein and Rautenbach \cite{HLR12} 
obtained an improved $\chi$-binding function $f(x)=\frac{3}{2}x$ for $(3P_1, K_4+P_1)$-free graphs.

An important subclass of $P_5$-free graphs is the class of $2P_2$-free graphs. It was known that
for any  $2P_2$-free graph it holds that $\chi\le \binom{\omega+1}{2}$ \cite{Wa80}. 
For a slightly larger class, namely $P_2+P_3$-free graphs, Bharathi and Choudum \cite{BC16} gave an $O(\omega^3)$
bound on $\chi$. 
Brause, Randerath, Schiermeyer and Vumar \cite{BRSV16} recently showed that 
$(P_5, butterfly)$-free graphs and $(P_5, hammer)$-free graphs, 
both of which are superclasses of $2P_2$-free graphs due to a recent structural result \cite{DSM16},
admit cubic and quadratic $\chi$-binding functions, respectively, where
a \emph{butterfly} is a graph isomorphic to $2P_2\vee P_1$ and a \emph{hammer}
is a graph on five vertices $\{a,b,c,d,e\}$ where $a,b,c,d$ in this order induces a $P_4$
and $e$ is adjacent to $a$ and $b$. 
It is not known whether any of these $\chi$-binding functions can be improved to linear.
Very recently, a linear $\chi$-binding function has been shown to exist for $(2P_2,H)$-free graphs when $H$
is one of $(P_1+P_2)\vee P_1$ (usually referred to as \emph{paw}), 
$P_4\vee P_1$ (usually referred to as \emph{gem})
 or $\overline{P_5}$ (usually referred to as \emph{house}) \cite{BRSV16}.
When $H$ is isomorphic to $C_4$, it was known \cite{BHPZ93} that
every such graph has $\chi\le \omega+1$; when $H$ is $P_2\vee 2P_1$
(usually referred to as \emph{diamond}),
it was known that $\chi\le \omega+3$. This bound in fact holds
for $(P_2+P_3,diamond)$-free graphs \cite{BC16}.
For more results on $\chi$-binding functions, we refer to a survey by Randerath and Schiermeyer \cite{RS04}.

\noindent {\bf Our Contributions.}
In this paper, we prove that $f(x)=\frac{3}{2}x$ is a $\chi$-binding function for $(P_6,C_4)$-free graphs.
This unifies several previous results
on the existence of linear $\chi$-binding functions for, e.g., $(2P_2,C_4)$-free graphs \cite{BHPZ93},
$(P_5,C_4)$-free graphs \cite{CKS07} and $(P_3+P_2,C_4)$-free graphs \cite{CK10}.
The graphs $C_5$, the Petersen graph, the Petersen graph with an additional universal vertex, and
all $4$-critical $(P_6,C_4)$-free graphs other than $K_4$ (see \cite{HH17}) show that our $\chi$-binding function
is optimal. 
On the other hand, there is an active research on classifying the complexity of coloring $(H_1,H_2)$-free graphs.
Despite much effort, the classification is far from being complete, see \cite{GolovachJPS16} for a summary of partial results.
The class of $(P_6,C_4)$-free graphs is one of the unknown cases.
(Note that the class of $(P_6,C_4)$-free graphs has unbounded clique-width \cite{DP16} and so we cannot
directly use the algorithm of Kobler and Rotics \cite{KR03} to conclude that coloring can be solved in polynomial
time for $(P_6,C_4)$-free graphs.)
Here we develop an $O(n^2m)$ 3/2-approximation algorithm for coloring $(P_6,C_4)$-free graphs.
This is the first approximation algorithm for coloring these graphs and could be viewed as a first step
towards a possible polynomial time algorithm for optimally coloring these graphs. 

The remainder of the paper is organized as follows. We present some preliminaries in \autoref{sec:pre}
and useful properties for $(P_6,C_4)$-free graphs that contain a $C_5$ in \autoref{sec:C5}.
We then prove a novel structure theorem for $(P_6,C_4)$-free graphs without clique cutsets in \autoref{sec:structure}.
Using this theorem we show in  \autoref{sec:3/2} that every $(P_6,C_4)$-free graph has chromatic number at most
$3/2$ its clique number. Finally, we turn our proof into a 3/2-approximation algorithm in \autoref{sec:alg}.

\section{Preliminaries}\label{sec:pre}

For general graph theory notation we follow \cite{BM08}.
Let $G=(V,E)$ be a graph. 
The \emph{neighborhood} of a vertex $v$, denoted by $N_G(v)$, is the set of neighbors of $v$.
For a set $X\subseteq V(G)$, let $N_G(X)=\bigcup_{v\in X}N_G(v)\setminus X$.
The \emph{degree} of $v$, denoted by $d_G(v)$, is equal to $|N_G(v)|$.
For $x\in V$ and $S\subseteq V$, we denote by $N_S(x)$ the set of neighbors of $x$ that are in $S$,
i.e., $N_S(x)=N_G(x)\cap S$.
For $X,Y\subseteq V$, we say that $X$ is \emph{complete} (resp. \emph{anti-complete}) to $Y$
if every vertex in $X$ is adjacent (resp. non-adjacent) to every vertex in $Y$.
If $X=\{x\}$, we write ``$x$ is complete (resp. anti-complete) to $Y$'' instead of ``$\{x\}$ is is complete (resp. anti-complete) to $Y$''.
A vertex subset $K\subseteq  V$ is a \emph{clique cutset} if $G-K$ has more components than $G$ and $K$
induces a clique. A vertex is \emph{universal} in $G$ if it is adjacent to all other vertices.
For $S\subseteq V$, the subgraph \emph{induced} by $S$, is denoted by $G[S]$.
A subset $M\subseteq V$ is a \emph{dominating set} if every vertex not in $M$ has a neighbor in $M$.
We say that $M$ is a \emph{module} if every vertex not in $M$ is either complete or anti-complete to $M$.

Let $u,v\in V$. We say that $u$ and $v$ are \emph{twins} if $u$ and $v$ are adjacent
and they have the same set of neighbors in $V\setminus \{u,v\}$. Note that the binary relation
of being twins on $V$ is an equivalence relation and so $V$ can be partitioned into equivalence 
classes $T_1,\ldots, T_r$ of twins. The \emph{skeleton} of $G$ is the subgraph induced by 
a set of $r$ vertices, one from each of $T_1,\ldots,T_r$.
A \emph{blow-up} of a graph $G$ is a graph $G'$ obtained by replacing each vertex $v$ of $G$
with a clique $K_v$  of size at least $1$ such that $K_v$ and $K_u$ are complete in $G'$ if $u$ and $v$ are adjacent in $G$,
and anti-complete otherwise. Since each equivalence class of twins is a clique and any two equivalence classes
are either complete or anti-complete, every graph is a blow-up of its skeleton.

A graph is \emph{chordal} if it does not contain any induced cycle of length at least four.
The following structure of $(P_6,C_4)$-free graphs discovered by Brandst{\"a}dt and Ho\`{a}ng \cite{BH07} is of particular
importance in our proofs below.
\begin{lemma}[Brandst{\"a}dt and Ho\`{a}ng \cite{BH07}]\label{lem:P6C4atom}
Let $G$ be a $(P_6,C_4)$-free graph without clique cutsets.  Then the following statements hold:
\begin{inparaenum}[(i)]
\item every induced $C_5$ is dominating;
\item If $G$ contains an induced $C_6$ which is not dominating, $G$ is the join of
a blow-up of the Petersen graph (\autoref{fig:counterexample}) and a (possibly empty) clique.
\end{inparaenum} 
\end{lemma}

\section{Structural Properties Around a $C_5$}\label{sec:C5}

In this section, we present structural properties of $(P_6,C_4)$-free graphs that contain a $C_5$.
Let $G=(V,E)$ be a graph and $H=1,2,\ldots,s,1$ be an induced cycle of $G$.
We partition $V\setminus V(H)$ into subsets with respect to $H$ as follows:
for any $X\subseteq V(H)$, we denote by $S(X)$ the set of vertices
in $V\setminus V(H)$ that have $X$ as their neighborhood among $V(H)$, i.e.,
\[S(X)=\{v\in V\setminus V(H): N_{V(H)}(v)=X\}.\]
For $0\le j\le |V(H)|$, we denote by $S_j$ the set of vertices in $V\setminus V(H)$ that have exactly $j$
neighbors among $V(H)$. Note that $S_j=\bigcup_{X\subseteq V(H): |X|=j}S(X)$.
We say that a vertex in $S_j$ is a \emph{$j$-vertex}. 
For simplicity, we shall write $S(1,2)$ for $S(\{1,2\})$ and $S(1,2,3)$ for $S(\{1,2,3\})$, etc.

Let $G$ be a $(P_6,C_4)$-free graph that contains an induced cycle $C=1,2,3,4,5$.
We partition $V(G)$ with respect to $C$. Then the following holds.
\begin{enumerate}[label=\bfseries (P\arabic*)]

\emitem {$V(G)=V(C)\cup  S_1\cup \bigcup_{i=1}^{5}S(i,i+1)\cup \bigcup_{i=1}^{5}S(i-1,i,i+1)\cup S_5$}\label{item:partition}

Suppose that $x\in V(G)\setminus V(C)$. Note that if $x$ is adjacent to $i$ and $i+2$ but not to $i+1$,
then $\{x,i,i+1,i+2\}$ induces a $C_4$. Thus, \ref{item:partition} follows. \e

\emitem {$S_5\cup S(i-1,i,i+1)$ is a clique.}\label{item:s5}

Suppose not. Let $u$ and $v$ be two non-adjacent vertices in $S_5\cup S(i-1,i,i+1)$.
Then $\{u,i-1,v,i+1\}$ induces a $C_4$, a contradiction.
\e

\emitem {$S(i-1,i,i+1)$ is anti-complete to $S(i+1,i+2,i+3)$.}\label{item:s3}

By symmetry, it suffices to prove \ref{item:s3} for $i=1$. Suppose that $u\in S(5,1,2)$
is adjacent to $v\in S(2,3,4)$. Then $\{5,4,v,u\}$ induces a $C_4$. \e

\emitem {$S(i,i+1)$ is complete to $S(i+1,i+2)$ and anti-complete to $S(i+3,i+4)$.
Moreover, if both $S(i,i+1)$ and $S(i+1,i+2)$ are not empty, then both sets are cliques.}\label{item:s2}

It suffices to prove \ref {item:s2} for $i=1$. Suppose that $u\in S(1,2)$ is not adjacent to $v\in S(2,3)$. Then
$u,1,5,4,3,v$ induces a $P_6$, a contradiction.  If $S(1,2)$ and $S(2,3)$ are not empty, then it follows from
the $C_4$-freeness of $G$ that both sets are cliques.
Similarly, if  $u\in S(1,2)$ is adjacent to $w\in S(3,4)$, then $\{2,3,w,u\}$ induces a $C_4$.  
\e

\emitem {$S(i)$ is anti-complete to $S(i+1)$, and is complete to $S(i+2)$.}\label{item:s1}

It suffices to prove the statement for $i=1$. If $u\in S(1)$ is adjacent to $v\in S(2)$, then $\{1,2,v,u\}$ induces a $C_4$,
a contradiction. Similarly, if $u\in S(1)$ is not adjacent to $w\in S(3)$, then $u,1,5,4,3,w$ induces a $P_6$. \e

\emitem {$S(i-1,i,i+1)$ is anti-complete to $S(i-2,i+2)$.}\label{item:s3antics2}.

By symmetry, it suffices to prove for $i=1$. If $t\in S(5,1,2)$ is adjacent to $d\in S(3,4)$,
then $\{4,5,t,d\}$ induces a $C_4$. \e

\item{$S(i)$ is anti-complete to $S(i+1,i+2,i+3)$.}\label{item:s1antics3}

By symmetry, it suffices to prove for $i=1$. If $u\in S(1)$ is adjacent to $t\in S(2,3,4)$,
then $\{1,u,t,2\}$ induces a $C_4$. \e

\emitem {One of $S(i)$ and $S(i+1,i+2)$ is empty.}\label{item:S1S2empty}

By symmetry, it suffices to show this for $i=1$. Suppose that $u\in S(1)$
and $v\in S(2,3)$. Then either $u,1,5,4,3,v$ induces a $P_6$ or $\{u,1,2,v\}$
induces a $C_4$, depending on whether $u$ and $v$ are adjacent. This is a contradiction. \e

\emitem {$S(i-2,i+2)$ is anti-complete to $S(j)$ if $j\neq i$.}\label{item:s1s2}

By symmetry, it suffices to prove for $i=1$. Let $x\in S(3,4)$. If $x$ has a neighbor $u\in S(2)$,
then $\{2,3,u,x\}$ induces a $C_4$. If $x$ has a neighbor $u\in S(3)$, then $u,x,4,5,1,2$
induces a $P_6$. 
This shows that $S(3,4)$ is anti-complete to $S(j)$ for all $j\neq 1$. \e
\end{enumerate}

\section{The structure of $(P_6,C_4)$-free atoms}\label{sec:structure}

A graph without clique cutsets is called an \emph{atom}.
We say that a vertex $v$ in $G$ is \emph{small} if $d_G(v)\le \frac{3}{2}\omega(G)-1$.
Our main result in this section is the following.

\begin{theorem}\label{thm:main}
Let $G$ be a $(P_6,C_4)$-free atom. Then one of the following is true:

$\bullet$ $G$ contains a small vertex,

$\bullet$ $G$ contains a universal vertex,

$\bullet$ $G$ is a blow-up of the Petersen graph (see \autoref{fig:counterexample}),

$\bullet$ $G$ is a blow-up of the graph $F$ (see \autoref{fig:counterexample}).
\end{theorem}

\begin{figure}[tb]
\centering
\begin{subfigure}{.5\textwidth}
\centering
\begin{tikzpicture}[scale=0.6]
\tikzstyle{vertex}=[draw, circle, fill=white!100, minimum width=4pt,inner sep=2pt]

\node[vertex] (v1) at (-1.5,2) {};
\node[vertex] (v2) at (1.5,2) {};
\node[vertex] (v3) at (4,0) {};
\node[vertex] (v4) at (1.5,-2) {};
\node[vertex] (v5) at (-1.5,-2) {};
\node[vertex] (v6) at (-4,0) {};
\draw (v1)--(v2)--(v3)--(v4)--(v5)--(v6)--(v1);

\node[vertex] (v14) at (-1.5,-0.5) {};
\draw (v14)--(v1) (v14)--(v4);
\node[vertex] (v25) at (1.5,-0.5) {};
\draw (v25)--(v2) (v25)--(v5);
\node[vertex] (v36) at (0,1) {};
\draw (v36)--(v3) (v36)--(v6);
\node[vertex] (c) at (0,0) {};
\draw (c)--(v14) (c)--(v25) (c)--(v36);

\node at (0,-3) {The Petersen graph};
\end{tikzpicture}
\end{subfigure}%
\begin{subfigure}{.5\textwidth}
\centering
\begin{tikzpicture}[scale=0.6]
\tikzstyle{vertex}=[draw, circle, fill=white!100, minimum width=4pt,inner sep=2pt]

\node[vertex] (u1) at (-1.5,2) {};
\node[vertex] (u2) at (1.5,2) {};
\node[vertex] (u3) at (4,0) {};
\node[vertex] (u4) at (1.5,-2) {};
\node[vertex] (u5) at (-1.5,-2) {};
\node[vertex] (u6) at (-4,0) {};
\draw (u1)--(u2)--(u3)--(u4)--(u5)--(u6)--(u1);

\node[vertex] (q6123) at (0,1) {};
\draw (q6123)--(u1)  (q6123)--(u2)  (q6123)--(u3)  (q6123)--(u6);
\node[vertex] (q2345) at (2,-0.5) {};
\draw (q2345)--(u2) (q2345)--(u3) (q2345)--(u4) (q2345)--(u5);
\node[vertex] (q4561) at (-2,-0.5) {};
\draw (q4561)--(u4)  (q4561)--(u5)  (q4561)--(u6)  (q4561)--(u1);

\node at (0,-3) {$F$};

\end{tikzpicture}
\end{subfigure}
\caption{Two smallest $(P_6,C_4)$-free atoms that do not contain any small vertex.}
\label{fig:counterexample}
\end{figure}
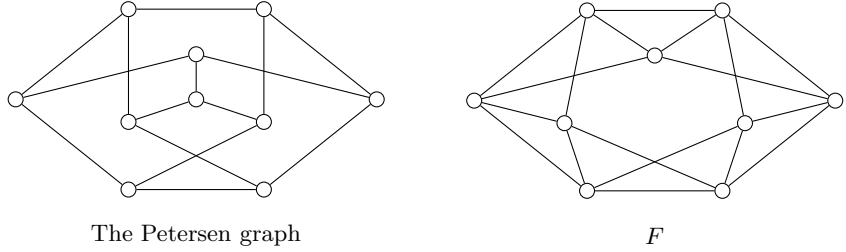

\noindent
To prove the above theorem, we shall prove a number of lemmas below.
The idea is that we assume the occurrence of some induced subgraph $H$ in $G$
and then argue that the theorem holds in this case. Afterwards, we can assume that
$G$ is $H$-free in addition to being $(P_6,C_4)$-free. We then pick a different induced subgraph as
$H$ and repeat. In the end, we are able to show that the theorem holds if $G$ contains a $C_5$ or $C_6$ 
(\autoref{lem:C6} and \autoref{lem:C5}). Therefore, the remaining case is that $G$ is chordal. In that case,
the theorem follows from a well-known fact \cite{Di61} that every chordal graph has a \emph{simplicial} vertex, that is, 
a vertex whose neighborhood
induces a clique. As straightforward as the approach sounds, the difficulty is that in order to eliminate $C_5$ and $C_6$
we have to eliminate two more special graphs $F_1$ and $F_2$ (\autoref{lem:F1} and \autoref{lem:F2}) and do it in the 
`right' order. We start with $F_1$.

\begin{figure}[tb]
\centering
\begin{subfigure}{.5\textwidth}
\centering
\begin{tikzpicture}[scale=0.7]
\tikzstyle{vertex}=[draw, circle, fill=white!100, minimum width=4pt,inner sep=2pt]

\node [vertex] (v1) at (0,3) {$1$};
\node [vertex] (v2) at (3,1) {$2$};
\node [vertex] (v3) at (1.5,-1) {$3$};
\node [vertex] (v4) at (-1.5,-1) {$4$};
\node [vertex] (v5) at (-3,1) {$5$};
\draw 
(v1)--(v2)--(v3)--(v4)--(v5)--(v1);

\node[vertex] (y) at (1.5,0.5) {$y$};
\draw (v2)--(y) (v3)--(y);

\node[vertex] (z) at (-1.5,0.5) {$z$};
\draw (v4)--(z) (v5)--(z);

\node [vertex] (x) at (0,0) {$x$};
\draw (x)--(v3)   (x)--(v4)  (x)--(y)  (x)--(z);

\node at (0,-2) {$F_1$};
\end{tikzpicture}
\end{subfigure}%
\begin{subfigure}{.5\textwidth}
\centering
\begin{tikzpicture}[scale=0.7]
\tikzstyle{vertex}=[draw, circle, fill=white!100, minimum width=4pt,inner sep=2pt]

\node [vertex] (v1) at (0,3) {$1$};
\node [vertex] (v2) at (3,1) {$2$};
\node [vertex] (v3) at (1.5,-1) {$3$};
\node [vertex] (v4) at (-1.5,-1) {$4$};
\node [vertex] (v5) at (-3,1) {$5$};
\draw 
(v1)--(v2)--(v3)--(v4)--(v5)--(v1);

\node[vertex] (y) at (1.5,0.5) {$y$};
\draw (v2)--(y) (v3)--(y);

\node[vertex] (z) at (-1.5,0.5) {$x$};
\draw (v4)--(z) (v5)--(z);

\node [vertex] (t) at (0,1.5) {$t$};
\draw (t)--(v5)   (t)--(v1)  (t)--(v2)  (t)--(y)  (t)--(z);

\node at (0,-2) {$F_2$};
\end{tikzpicture}
\end{subfigure}
\caption{Two special graphs $F_1$ and $F_2$.}
\label{fig:F1F2}
\end{figure}
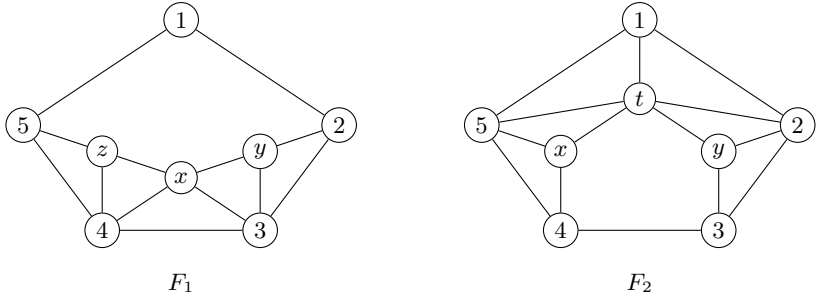

\begin{lemma}\label{lem:F1}
If a $(P_6,C_4)$-free atom $G$ contains $F_1$ (see \autoref{fig:F1F2}), then $G$ contains a small vertex or a universal vertex.
\end{lemma}

\begin{proof}
Let $G$ be a $(P_6,C_4)$-free atom that contains an induced subgraph $H$ that is isomorphic to $F_1$
with $V(H)=\{1,2,3,4,5,x,y,z\}$ where $1,2,3,4,5,1$ induces the \emph{underlying} five-cycle $C$
of $F_1$ and $x$ is adjacent to $3$ and $4$, $y$ is adjacent to $2$ and $3$, 
$z$ is adjacent to $4$ and $5$, and $x$ is adjacent to $y$ and $z$, see \autoref{fig:F1F2}. 
We partition $V(G)$ with respect to $C$.
We choose the copy of $H$ in $G$ such that $|S_2|$ maximized.
Note that $x\in S(3,4)$, $y\in S(2,3)$ and $z\in S(4,5)$.
All indices below are modulo $5$. Since $G$ is an atom, it follows from
\autoref{lem:P6C4atom} that $S_0=\emptyset$. 

\begin{enumerate}[label=\bfseries (\arabic*)]

\emitem {$S_2=S(2,3)\cup S(3,4)\cup S(4,5)$}\label{item:s2empty}

Recall that $x\in S(3,4)$, $y\in S(2,3)$ and $z\in S(4,5)$. By symmetry, it suffices to
show that $S(1,2)=\emptyset$. Suppose that $S(1,2)$ contains one vertex, say $s$.
Then $s$ is not adjacent to $x$ and $z$, and $x$ and $z$ are adjacent by \ref{item:s2}.
This implies that $5,z,x,3,2,s$ induces a $P_6$, a contradiction. \e

\emitem {$S_1=\emptyset$.}\label{item:s1empty}

Since $S(2,3)$, $S(3,4)$ and $S(4,5)$ are not empty, the statement follows directly from \ref{item:S1S2empty}.
\e

\emitem {$S_5$ is either complete or anti-complete to $S_2$.} \label{item:s5antis2}

Let $u\in S_5$ be an arbitrary vertex.
Note that any $x'\in S(3,4)$ and any $z'\in S(4,5)$ are adjacent by \ref{item:s2}.
Consider the induced six-cycle $C'=z',5,1,2,3,x',z'$. Since $u$
is adjacent to $5,1,2,3$, it follows from the $C_4$-freeness of $G$ 
that $u$ is either complete or anti-complete to $\{x',z'\}$. 
Similarly,  $u$ is either complete or anti-complete to $\{x',y'\}$
for any $x'\in S(3,4)$ and $y'\in S(2,3)$.
This implies that $u$ is either complete or anti-complete to $S_2$.
\e

\emitem {$S(5,1,2)$ is anti-complete $S_2$.} \label{item:s3s21}

Let $t\in S(5,1,2)$ be an arbitrary vertex. By \ref{item:s3antics2}, $S(5,1,2)$ is anti-complete to $S(3,4)$.  
Suppose that $t$ has a neighbor in $S(2,3)\cup S(4,5)$, say  $y'\in S(2,3)$.
Note that $y'$ is adjacent to $x$ and non-adjacent to $z$ by \ref{item:s2}, and that $t$ is not adjacent
to $x$ by \ref{item:s3antics2}.
Then either $1,t,y',3,4,z$ induces a $P_6$ or $\{t,z,x,y'\}$ induces a $C_4$, depending on
whether $t$ and $z$ are adjacent.  \e

\emitem {$S(4,5,1)$ is complete to $S(4,5)$. By symmetry, $S(1,2,3)$ is complete to $S(2,3)$.} \label{item:s3s22}

Let $t\in S(4,5,1)$ and $z'\in S(4,5)$ be two arbitrary vertices. 
Suppose that $t$ and $z'$ are not adjacent. 
Note that $t$ is not adjacent to $y$ by \ref{item:s3antics2}.
Then
either $t,5,z',x,y,2$ induces a $P_6$ or $\{t,5,z',x\}$ induces a $C_4$,
depending on whether $t$ and $x$ are adjacent.
\e

\emitem {$S(2,3,4)$ is complete to $S(3,4)$. By symmetry, $S(3,4,5)$ is complete to $S(3,4)$.} \label{item:s3s23}

Let $t\in S(2,3,4)$ and $x'\in S(3,4)$ be two arbitrary vertices. 
Suppose that $t$ and $x'$ are not adjacent. 
Note that $x'$ is adjacent to $z$ by \ref{item:s2}.
Then either $t,3,x',z,5,1$ induces 
a $P_6$ or $\{t,3,x',z\}$ induces a $C_4$, depending on whether $t$ and $z$ are adjacent.
\e

\emitem {$S(2,3,4)$ is complete to $S(2,3)$. By symmetry, $S(3,4,5)$ is complete to $S(4,5)$.} \label{item:s3s24}

Let $t\in S(2,3,4)$ and $y'\in S(2,3)$ be two arbitrary vertices. 
By \ref{item:s2} and \ref{item:s3s23}, $x$ is adjacent to both $t$ and $y'$.
So, $t$ and $y'$ are adjacent, for otherwise $\{t,x,y',2\}$ induces a $C_4$.
\e

\emitem {$S(5,1,2)$ is complete to $S(1,2,3)$ and $S(4,5,1)$.} \label{item:s3s31}

By symmetry, it suffices to show that $S(5,1,2)$ is complete to $S(4,5,1)$.
Suppose that $s\in S(5,1,2)$ is not adjacent to $t\in S(4,5,1)$. 
Note that $s$ is not adjacent to $y$ by \ref{item:s3s21}, and that $t$ is not adjacent to $y$ by \ref{item:s3antics2}. Then.
Then $s,1,t,4,3,y$ induces a $P_6$, and this is a contradiction. 
\e

\emitem {Let $s\in S(3,4,5)$ and $t\in S(4,5,1)$ such that $s$ and $t$ are not adjacent. Then $t$
is anti-complete to $S(3,4)$ and $s$ is complete to $S(2,3)$.} \label{item:s3s32}

Let $x'\in S(3,4)$ be an arbitrary vertex.
First,  $x'$ and $s$ are adjacent by \ref{item:s3s23}. 
Moreover, $x'$ and $t$ are not adjacent, for otherwise $\{5,t,x',s\}$ induces a $C_4$.
This proves the first part of  \ref{item:s3s32}. Now let $y'\in S(2,3)$ be an arbitrary
vertex. By \ref{item:s2}, $y'$ is adjacent to $x$. 
Moreover, $y'$ is not adjacent to $t$ by \ref{item:s3antics2}, and $s$ is adjacent to $x$ by \ref{item:s3s23}.
If $s$ and $y'$ are not adjacent, then $t,5,s,x,y',2$
induces a $P_6$, a contradiction. This shows that $s$ is complete to $S(2,3)$.
\e

\emitem {Let $s\in S(2,3,4)$ and $t\in S(3,4,5)$ such that $s$ and $t$ are not adjacent. 
Then $s$ is anti-complete to $S(4,5)$ and $t$ is anti-complete to  $S(2,3)$.} \label{item:s3s33}

Let $z'\in S(4,5)$ be an arbitrary vertex. By  \ref{item:s3s24}, $t$ is adjacent to $z'$.
If $s$ and $z'$ are adjacent, then $\{s,z',t,3\}$ induces a $C_4$, a contradiction.
This proves that $s$ is anti-complete to $S(4,5)$. By symmetry, $t$ is anti-complete
to $S(2,3)$.
\e
\end{enumerate}

\noindent
We  distinguish two cases depending on whether $S_5$ is empty.

\noindent {\bf Case 1.} $S_5$ contains a vertex $u$. 
By \ref{item:s5antis2}, $u$ is either complete or anti-complete to $S_2$.
If $u$ is complete to $S_2$, then $u$ is a universal in $G$ by \ref{item:partition}, \ref{item:s5} and \ref{item:s1empty},
and we are done.  So, we assume that $u$ is anti-complete to $S_2$.
We prove some additional 
properties of the graph with the existence of $u$.

\begin{enumerate}[label=\bfseries (\alph*)]
\emitem {$S(3,4,5)$ is anti-complete to $S(2,3)$. By symmetry, $S(2,3,4)$ is anti-complete to $S(4,5)$.} \label{item:s3s25}

Let $t\in S(3,4,5)$ and $y'\in S(2,3)$ be two arbitrary vertices.
Suppose that $t$ and $y'$ are adjacent. By \ref{item:s5} and \ref{item:s5antis2}, $u$ is adjacent to $t$
but not adjacent to $y'$. Then $\{t,u,2,y'\}$ induces a $C_4$, a contradiction. This proves the claim. 
\e

\emitem {$S(4,5,1)$ is anti-complete to $S(3,4)$. By symmetry, $S(1,2,3)$ is anti-complete to $S(3,4)$.} \label{item:s3s26}

Let $t\in S(4,5,1)$ and $x'\in S(3,4)$ be two arbitrary vertices. 
 By \ref{item:s5} and \ref{item:s5antis2}, $u$ is adjacent to $t$
but not adjacent to $x'$.  If $t$ and $x'$ are adjacent, 
then $\{t,u,3,x'\}$ induces a $C_4$, a contradiction.
\e

\emitem {$S(2,3,4)$ is complete to $S(3,4,5)$.} \label{item:s3s34}

Let $s\in S(2,3,4)$ and $t\in S(3,4,5)$ be two arbitrary vertices.
Then $x$ is adjacent to both $s$ and $t$ by \ref{item:s3s23}.
By \ref{item:s5} and \ref{item:s5antis2}, $u$ is adjacent to $s$ and $t$
but not adjacent to $x$. If $s$ and $t$ are not adjacent, then $\{x,s,u,t\}$
induces a $C_4$. 
\e

\emitem {$S(4,5,1)$ is complete to $S(3,4,5)$. By symmetry, $S(1,2,3)$ is complete to $S(2, 3,4)$. } \label{item:s3s35}

This follows directly from \ref{item:s3s25} and \ref{item:s3s32}.
\e
\end{enumerate}

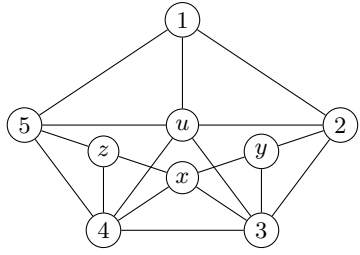
\begin{figure}[tb]
\centering
\begin{tikzpicture}[scale=0.7]
\tikzstyle{vertex}=[draw, circle, fill=white!100, minimum width=4pt,inner sep=2pt]

\node [vertex] (v1) at (0,3) {$1$};
\node [vertex] (v2) at (3,1) {$2$};
\node [vertex] (v3) at (1.5,-1) {$3$};
\node [vertex] (v4) at (-1.5,-1) {$4$};
\node [vertex] (v5) at (-3,1) {$5$};
\node[vertex] (u) at (0,1) {$u$};
\draw 
(v1)--(v2)--(v3)--(v4)--(v5)--(v1);

\node[vertex] (y) at (1.5,0.5) {$y$};
\draw (v2)--(y) (v3)--(y);

\node[vertex] (z) at (-1.5,0.5) {$z$};
\draw (v4)--(z) (v5)--(z);

\node [vertex] (x) at (0,0) {$x$};
\draw (x)--(v3)   (x)--(v4)  (x)--(y)  (x)--(z);

\draw
(u)--(v1) (u)--(v2) (u)--(v3) (u)--(v4) (u)--(v5);
\end{tikzpicture}
\caption{The graph $F_3$.}
\label{fig:F3}
\end{figure}

\noindent
Note that $S(4,5,1)$ is anti-complete to $S(2,3)$ and that $S(1,2,3)$ is anti-complete
to $S(4,5)$ by \ref{item:s3antics2}.
It  follows from \ref{item:s5}-\ref{item:s1}, \ref{item:s2empty}-\ref{item:s3s33}, and \ref{item:s3s25}-\ref{item:s3s35} that
$G$ is a blow-up of a special graph $F_3$ (see \autoref{fig:F3}). We denote by $Q_v$ the clique that $v\in V(F_3)$ is blown into.
Suppose first that $|Q_2|\le \omega(G)/2$.
Note that $N_G(y)=(Q_y\setminus \{y\})\cup Q_x\cup Q_2\cup Q_3$.
Since $(N_G(y)\setminus Q_2)\cup \{y\}$ is a clique, it follows that $|N_G(y)\setminus Q_2|\le \omega(G)-1$.
Therefore, $d_G(y)\le \omega(G)-1+|Q_2|\le \omega(G)-1+\omega(G)/2=\frac{3}{2}\omega(G)-1$. 
Now suppose that $|Q_2|> \omega(G)/2$. This implies that $|Q_1\cup Q_u|<\omega(G)/2$.
Note that $(N_G(5)\cup \{5\})\setminus (Q_1\cup Q_u)$ is a clique. 
Therefore, $d_G(5)\le \omega(G)-1+|Q_1\cup Q_u|\le \frac{3}{2}\omega(G)-1$.

\noindent {\bf Case 2.} $S_5$ is empty. 

\begin{enumerate}[label=\bfseries (\alph*)]
\emitem {$S(1,2,3)$ is complete to $S(2,3,4)$. By symmetry, $S(4,5,1)$ is complete to  $S(3,4,5)$.} \label{item:s3s36}

Suppose that $s\in S(2,3,4)$ and $r\in S(1,2,3)$ are not adjacent. By  \ref{item:s3s32} and \ref{item:s3s22},
$r$ is complete to $S(2,3)$ and anti-complete to $S(3,4)$. 
Moreover, $r$ is anti-complete to $S(4, 5)$ by \ref{item:s3antics2}.
Note that $V(H)\setminus \{2\}\cup \{r\}$
also induces a subgraph $H'$ that is isomorphic to $F_1$ whose underlying five-cycle is $C'=C\setminus \{2\}\cup \{r\}$.
Clearly, $s$ is adjacent to exactly two vertices on $C'$. Therefore, 
the number of $2$-vertices with respect to $C'$ is more than that with respect to $C$, and this contradicts the
choice of $H$.
\e

\emitem {$S(2,3,4)$ is complete to $S(3,4,5)$.} \label{item:s3s37}

Suppose that $s\in S(2,3,4)$ and $r\in S(3,4,5)$ are not adjacent. 
By \ref{item:s3s23} and \ref{item:s3s24}, $s$ is complete to $S(2,3)\cup S(3,4)$.  
By  \ref{item:s3s33}, $s$ is anti-complete to $S(4,5)$.
Note that $V(H)\setminus \{3\}\cup \{s\}$
also induces a subgraph $H'$ that is isomorphic to $F_1$ whose underlying five-cycle is $C'=C\setminus \{3\}\cup \{s\}$.
Clearly, $r$ is adjacent to exactly two vertices in $C'$. Therefore, 
the number of $2$-vertices with respect to $C'$ is more than that with respect to $C$, and this contradicts the
choice of $H$.
\e
\end{enumerate}

\noindent
By \ref{item:s3s36}, \ref{item:s3s37}, \ref{item:s3s21}-\ref{item:s3s33},
$S(i-1,i,i+1)$ is complete to $S(i,i+1,i+2)$,  $S(2,3)$ is complete to $S(1,2,3)\cup S(2,3,4)$
and anti-complete to $S(4,5,1)\cup S(5,1,2)$, $S(4,5)$ is complete to $S(3,4,5)\cup S(4,5,1)$
and anti-complete to $S(5,1,2)\cup S(1,2,3)$, and $S(3,4)$ is complete to $S(2,3,4)\cup S(3,4,5)$
and anti-complete to $S(5,1,2)$, see \autoref{fig:noS5}.

\begin{figure}[tb]
\center
\begin{tikzpicture}[scale=0.5]
\tikzstyle{vertex}=[draw, circle, fill=white!100, minimum width=4pt,inner sep=2pt]
\tikzstyle{set}=[draw, circle, minimum width=2pt, inner sep=2pt]

\node [set, blue] (v1) at (0,3) {$1$};
\node [set, blue] (v2) at (6,0.5) {$2$};
\node [set, blue] (v3) at (3,-3) {$3$};
\node [set,blue] (v4) at (-3,-3) {$4$};
\node [set,blue] (v5) at (-6,0.5) {$5$};

\node [set] (x) at (0,-1.5) {$S(3,4)$};
\node[set] (y) at (2.7,0) {$S(2,3)$};
\node[set] (z) at (-2.7,0) {$S(4,5)$};

\node[set] (t1) at (0,6) {$S(5,1,2)$};
\node[set] (t2) at (9, 0) {$S(1,2,3)$};
\node[set] (t5) at (-9,0) {$S(4,5,1)$};
\node[set] (t3) at (5,-5) {$S(2,3,4)$};
\node[set] (t4) at (-5,-5) {$S(3,4,5)$};

\draw [blue]
(v1)--(v2)--(v3)--(v4)--(v5)--(v1);
\draw[ultra thick]
(x)--(v3)
(x)--(v4)
(x)--(y)
(v2)--(y)
(v3)--(y)
(x)--(z)
(v4)--(z)
(v5)--(z)

(t1)--(v5) (t1)--(v1) (t1)--(v2)
(t2)--(v1) (t2)--(v2) (t2)--(v3)  (t2)--(y) (t2)--(t1)
(t5)--(v4) (t5)--(v5) (t5)--(v1) (t5)--(z) (t5)--(t1)
(t3)--(v2) (t3)--(v3) (t3)--(v4) (t3)--(y) (t3)--(x) (t3)--(t2)
(t4)--(v3) (t4)--(v4) (t4)--(v5) (t4)--(z) (t4)--(x) (t4)--(t5) (t4)--(t3);       

\draw
(t2)--(x)
(t5)--(x);
\draw[bend left]
(t3)--(z)
(t4)--(y);
\end{tikzpicture}
\caption{The structure of $G$. A thick line between two sets represents that the two sets are complete to each other, and
a thin line represents that the edges between the two sets can be arbitrary.  Two sets are anti-complete
if there is no line between them.}\label{fig:noS5}
\end{figure}
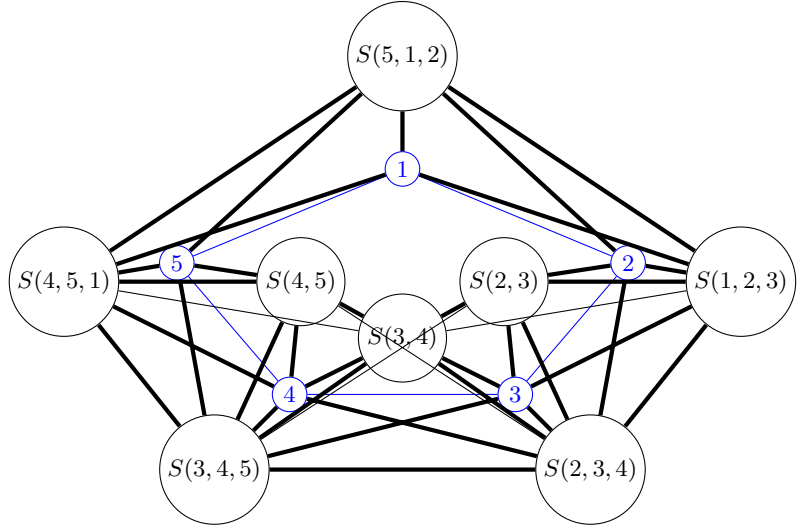 

Let $Q_i=S(i-1,i,i+1)\cup \{i\}$ for $i\in \{1,2,3,4,5\}$.
Suppose first that $|Q_2|\le \omega(G)/2$. Then
$d_G(1)=|S(5,1,2)\cup Q_5|+|Q_2|\le \omega(G)-1+\omega(G)/2=\frac{3}{2}\omega(G)-1$.
Now suppose that $|Q_2|> \omega(G)/2$. This implies that $|Q_1|<\omega(G)/2$.
Therefore, $d_G(5)=|S(4,5,1)\cup Q_4\cup S(4,5)|+|Q_2|\le \omega(G)-1+\omega(G)/2=\frac{3}{2}\omega(G)-1$.
This completes the proof of Case 2.

\end{proof}

\begin{lemma}\label{lem:C6}
If a $(P_6,C_4,F_1)$-free atom $G$ contains a $C_6$, then one of the following is true:

$\bullet$ $G$ contains a small vertex,

$\bullet$ $G$ contains a universal vertex,

$\bullet$ $G$ is a blow-up of the Petersen graph (see \autoref{fig:counterexample}),

$\bullet$ $G$ is a blow-up of the graph $F$ (see \autoref{fig:counterexample}).
\end{lemma}

\begin{proof} Let $C=1,2,3,4,5,6,1$ be an induced six-cycle of $G$. We partition $V(G)$
with respect to $C$. If $C$ is not dominating, then it follows from \autoref{lem:P6C4atom} that $G$ is
the join of a blow-up of the Petersen graph and a (possibly empty) clique. This implies that
either $G$ contains a universal vertex or $G$ is a blow-up of the Petersen graph.
(Note that the Petersen graph does not contain any small vertex: every vertex has degree $3>\frac{3}{2}\times 2-1$).

In the following, we assume that $C$ is dominating, i.e., $S_0= \emptyset$. 
All indices below are modulo $6$.
It is straightforward to verify (by the fact that $G$ is $(C_4,P_6)$-free) that
\[V(G)=V(C)\cup \bigcup_{i}S(i,i+3)\cup \bigcup_{i}S(i-1,i,i+1)\cup
\bigcup_{i}S(i-1,i,i+1,i+2)\cup S(C).\]

\begin{enumerate}[label=\bfseries (\arabic*)]

\emitem {Let $X,Y\subseteq \{1,2,3,4,5,6\}$ such that $X\cap Y$ contains two non-adjacent vertices $i$ and $j$.
Then $S(X)$ is complete to $S(Y)$. In particular, if $X=Y$ then $S(X)$ is a clique.}\label{cla:complete}

Let $x\in S(X)$ and $y\in S(Y)$ be two arbitrary vertices.
If $x$ and $y$ are not adjacent, then $\{i,j,x,y\}$  induces a $C_4$, a contradiction.
\e

\emitem {Let $X,Y\subseteq \{1,2,3,4,5,6\}$ such that there exists an index $i\in \{1,2,3,4,5,6\}$ 
with $i\in X\setminus Y$ and $i+1\in Y\setminus X$.
Then $S(X)$ is anti-complete $S(Y)$.}\label{cla:anti-complete}

Let $x\in S(X)$ and $y\in S(Y)$ be two arbitrary vertices.
If $x$ and $y$ are adjacent, then $\{i,i+1,x,y\}$  induces a $C_4$, a contradiction. \e

\end{enumerate}

By \ref{cla:complete}, $S(i,i+3)$, $S(i-1,i,i+1)$ and $S(i-1,i,i+1,i+2)$ is a clique for each $i$, and every vertex
in $S(C)$ is a universal vertex in $G$. 
So, we may assume that $S(C)=\emptyset$ otherwise we are done.
We now explore the adjacency between different $S(X)$.
First of all, \ref{cla:anti-complete} implies that $S(1,4)$, $S(2,5)$ and $S(3,6)$ are pairwise
anti-complete to each other. Secondly, $S(1,4)$ is anti-complete to $S(i-1,i,i+1)$ for $i=2,3,5,6$ by \ref{cla:anti-complete}.
Suppose that $x\in S(1,4)$ is not adjacent to $y\in S(6,1,2)$. Then $2,y,6,5,4,x$ induces a $P_6$, a contradiction.
This shows that $S(1,4)$ is complete to $S(6,1,2)$ and hence to $S(3,4,5)$ by symmetry.
Similarly,  $S(1,4)$ is anti-complete to $S(i-1,i,i+1,i+2)$ for $i=1,3,4,6$ by \ref{cla:anti-complete}
and complete to  $S(i-1,i,i+1,i+2)$ for $i=2,5$ by \ref{cla:complete}.
So, $S(1,4)$ is a module of $G$. By symmetry, $S(2,5)$ and $S(3,6)$ are also modules.

We now show that each $S(i-1,i,i+1,i+2)$ is a module.  By \ref{cla:complete} and  \ref{cla:anti-complete},
$S(i-1,i,i+1,i+2)$ is complete to $S(j-1,j,j+1,j+2)$ if $|i-j|=1$ or  $|i-j|=3$, and anti-complete to $S(j-1,j,j+1,j+2)$ if $|i-j|=2$.
Since $S(i,i+3)$ is a module for each $i$, it remains to show that $S(i-1,i,i+1,i+2)$ is either complete or anti-complete
to $S(X)$ with $|X|=3$. By symmetry, it suffices to consider $S(1,2,3,4)$.
First, note that $S(1,2,3,4)$ is complete to $S(1,2,3)\cup S(2,3,4)$ by \ref{cla:complete},
and anti-complete to $S(4,5,6)\cup S(5,6,1)$ by \ref{cla:anti-complete}.
Suppose now that some vertex $x\in S(1,2,3,4)$ is not adjacent to some vertex $y\in S(6,1,2)$.
Then $C\cup \{x,y\}$ induces a subgraph isomorphic to $F_1$, contradicting 
our assumption that $G$ is $F_1$-free.
 Therefore, $ S(1,2,3,4)$ is complete to $S(6,1,2)$ and hence to $S(3,4,5)$ by symmetry.
Thus, $S(1,2,3,4)$ is indeed a module of $G$.

We show next that each $S(i-1,i,i+1)$ is a module. It remains to show that  $S(i-1,i,i+1)$ is either complete
or anti-complete to  $S(j-1,j,j+1)$ for $j\neq i$. It suffices to consider $S(1,2,3)$. 
If some vertex $x\in S(1,2,3)$ is not adjacent to some vertex $y\in S(2,3,4)$, then $x,1,6,5,4,y$ induces 
a $P_6$, a contradiction. So, $S(1,2,3)$ is complete to $S(2,3,4)$ and hence to $S(6,1,2)$ by symmetry.
Moreover, $S(1,2,3)$ is anti-complete to $S(4,5,6)$ by  \ref{cla:anti-complete}.
Suppose now that $x\in S(1,2,3)$ is adjacent to $y\in S(3,4,5)$. Then $C\cup \{x,y\}$ induces a subgraph
isomorphic to $F_1$,contradicting 
our assumption that $G$ is $F_1$-free.
Therefore, $S(1,2,3)$ is anti-complete to $S(3,4,5)$ and to $S(5,6,1)$. 
This shows that $S(1,2,3)$ is indeed a module of $G$.

It follows from the adjacency between $S(X)$ and $S(Y)$ for any $X,Y\subseteq \{1,2,3,4,5,6\}$ that
$M_i=S(i-1,i,i+1)\cup \{i\}$ is a module in $G$. 
Now we show that either $G$ contains a small vertex or $G$ is a blow-up of $F$ (see \autoref{fig:counterexample}).

\noindent {\bf Case 1.} $S_4\neq \emptyset$. By symmetry, we assume that $S(6,1,2,3)$
contains a vertex $x$.

\begin{enumerate}[label=\bfseries (\alph*)]

 \emitem {$S(1,4)=S(2,5)=\emptyset$.} \label{item:14and25} 
 
 By symmetry, it suffices to prove for $S(1,4)$. Suppose that $S(1,4)$ contains a vertex $y$.
 Then either $\{x,3,4,y\}$ induces a $C_4$ or $2,x,6,5,4,y$ induces a $P_6$, depending
 on whether $x$ and $y$ are adjacent. 
 \e
 
\emitem {One of $S(1,2,3,4)$ and $S(4,5,6,1)$ is empty. By symmetry, one of $S(2,3,4,5)$ and $S(5,6,1,2)$ is empty. } \label{item:14}

Suppose that neither $S(1,2,3,4)$ nor $S(4,5,6,1)$ is empty, say $y\in S(1,2,3,4)$ and $z\in S(4,5,6,1)$.
Then $y$ is adjacent to both $x$ and $z$, and $x$ and $z$ are not adjacent. Now $\{6,x,y,z\}$ induces a $C_4$,
a contradiction.
\e

\emitem {If $S(3,4,5,6)\neq \emptyset$, then $S_4=S(6,1,2,3)\cup S(3,4,5,6)$.}\label{item:36}

Let $z\in S(3,4,5,6)$. By symmetry, it suffices to show that $S(1,2,3,4)=\emptyset$. Suppose that
$y\in S(1,2,3,4)$. Then $x$ is adjacent to both $y$ and $z$, and $y$ and $z$ are not adjacent. Now $\{x,y,4,z\}$ induces a $C_4$. \e

\end{enumerate}

We now show that if $S(4,5,6,1)=\emptyset$, then either $1$ or $4$ is small.
Recall that $M_i=S(i-1,i,i+1)\cup \{i\}$ for $1\le i\le 6$ is a module of $G$,
and that $S(C)=\emptyset$.
By our assumption that $S(4,5,6,1)=\emptyset$ and \ref{item:14and25}, 
$N_{G}(4)=S(1,2,3,4)\cup S(2,3,4,5)\cup S(3,4,5,6)\cup M_3\cup S(3,4,5)\cup M_5$.
By the fact that $M_i$ is a module of $G$ and \ref{item:36},
$N_G(4)\setminus M_5$ is a clique. If $|M_5|\le \omega(G)/2$, then
$d_G(4)\le \frac{3}{2}\omega(G)-1$. So, we may assume that $|M_5|>\omega(G)/2$.
We then consider the vertex $1$. Note that 
$N_G(1)=Q_1\cup Q_2$, where $Q_1=S(6,1,2,3)\cup S(1,2,3,4)\cup M_2\cup S(6,1,2)$
and $Q_2=S(5,6,1,2)\cup M_6$ are cliques.
Moreover, $Q_2\cup M_5$ is a clique, and so $|Q_2|<\omega(G)/2$.
This implies that $d_G(1)\le \frac{3}{2}\omega(G)-1$. 
Similarly, if $S(2,3,4,5)=\emptyset$, then either $2$ or $5$ is small.

Now let $y\in S(2,3,4,5)$ and $z\in S(4,5,6,1)$. 
By \ref{item:14} and \ref{item:36}, $S(1,2,3,4)=S(3,4,5,6)=S(5,6,1,2)=\emptyset$.
If $S(3,6)$ contains a vertex $w$, then $w,3,2,1,z,5$ induces a $P_6$.
This means that $S(3,6)=\emptyset$. Recall that $S(1,4)=S(2,5)=\emptyset$ by \ref{item:14and25}.
Therefore, $G$ is a blow-up of $F$.
(Note that $F$ does not contain any small vertex: every vertex has degree $4>\frac{3}{2}\times 3-1$). This completes the proof of Case 1.

\noindent {\bf Case 2.} $S_4=\emptyset$.
Suppose first that $S(3,6)=\emptyset$. 
Since $M_1\cup M_2$ is a clique, either $M_1$ or $M_2$ has size
at most $\omega(G)/2$. If $|M_1|\le \omega(G)/2$, 
then $d_G(6)=|(M_6\setminus \{6\})\cup M_5|+|M_1|\le \omega(G)-1+\omega(G)/2=\frac{3}{2}\omega(G)-1$.
Similarly, if $|M_2|\le \omega(G)/2$, then $d_G(3)\le \frac{3}{2}\omega(G)-1$.
This  proves that either $3$ or $6$ is small.
So we assume that $S(3,6)$ contains a vertex $z$. By symmetry, 
we can assume that $S(1,4)$ contains a vertex $x$ and $S(2,5)$
contains a vertex $y$. 
Note that $N_G(x)=M_1\cup M_4$. If $|M_1|\le \omega(G)/2$,
then $d_G(x)\le |M_1|+|M_4|\le \omega(G)-1+\omega(G)/2=\frac{3}{2}\omega(G)-1$.
So, $|M_1|>\omega(G)/2$. This implies that $|M_2|\le \omega(G)/2$ since $M_1\cup M_2$ is a clique.
Then $d_G(y)=|M_2|+|M_5|\le \frac{3}{2}\omega(G)-1$.
 This completes the proof of Case 2.

This completes the proof of the lemma.
\end{proof}

\begin{lemma}\label{lem:F2}
If a $(P_6,C_4,C_6)$-free atom $G$ contains an $F_2$ (see \autoref{fig:F1F2}), then $G$ contains a small vertex.
\end{lemma}

\begin{proof} Let $G$ be a $(P_6,C_4,C_6)$-free atom that contains an induced subgraph 
$H$ that is isomorphic to $F_2$ with $V(H)=\{1,2,3,4,5,t,x,y\}$ such that $1,2,3,4,5,1$
induces the \emph{underlying} five-cycle $C$, and $t$ is adjacent to $5$, $1$ and $2$, $x$ is adjacent to $4,5$ and $y$
is adjacent to $2$ and $3$. Moreover, $t$ is adjacent to both $x$ and $y$, see \autoref{fig:F1F2}.
We partition $V(G)$ with respect to $C$. 
We choose the copy of $H$ in $G$ such that $C$ has $|S_2|$ maximized.
Note that $x\in S(4,5)$, $y\in S(2,3)$ and $t\in S(5,1,2)$.
Since $S(2,3)$ and $S(4,5)$ are not empty, it follows from \ref{item:S1S2empty} that $S_1 = S(2)\cup S(5)$.
If both $S(2)$ and $S(5)$ are not empty, say $u\in S(2)$ and $v\in S(5)$, then $u$ and $v$ are adjacent
by \ref{item:s1}, and so $u,2,3,4,5,v,u$ induces a $C_6$, contradicting our assumption that
$G$ is $C_6$-free. This shows that $S_1 = S(i)$ for some $i\in \{2,5\}$.
Now we argue that  $S_2=S(2,3)\cup S(4,5)$.
If $S(3,4)$ contains a vertex $z$, then $z$ is adjacent to $x$ and $y$
by \ref{item:s2}. This implies that either $\{t,5,4,z\}$ or $\{t,x,z,y\}$ induces a $C_4$, depending
on whether $t$ and $z$ are adjacent. So,  $S(3,4)=\emptyset$.
If $S(1,2)$ contains a vertex $z$,
then $z$ is adjacent to $y$ and so $1,z,y,3,4,5,1$ induces a $C_6$, a contradiction. This
shows that $S(1,2)=\emptyset$. By symmetry, $S(5,1)=\emptyset$.

\begin{enumerate}[label=\bfseries (\arabic*)]

\emitem {Each vertex in $S(5,1,2)$ is either complete or anti-complete to $S_2$.} \label{item:s512}

Let $t'\in S(5,1,2)$ be an arbitrary vertex. Suppose that $t'$ has a neighbor, say $x'$, in $S(4,5)$.
If $t'$ is not adjacent to a vertex $y'\in S(2,3)$, then $1,t',x',4,3,y'$ induces a $P_6$, 
since $x'$ and $y'$ are not adjacent to by \ref{item:s2}.
This shows that $t'$ is complete to $S(2,3)$.  In particular, $t'$ is adjacent to $y$.
Applying the same argument we conclude that $t'$ is also complete to $S(4,5)$. 
By symmetry, if $t'$ has a neighbor in $S(2,3)$, then $t'$ is also complete to
$S_2$.
\e

\emitem {$S(2,3)$ and $S(4,5)$ are cliques.} \label{item:s2clique}

By \ref{item:s512}, $t$ is complete to $S_2$. 
If $S(2,3)$ contains two non-adjacent vertices $y'$ and $y''$, then $\{t,y',y'',3\}$
induces a $C_4$, a contradiction. Therefore, $S(2,3)$ is a clique.
By symmetry, so is $S(4,5)$. 
\e

\emitem {Each vertex in $S(3,4,5)\cup S(4,5,1)$ is either complete or anti-complete to $S(4,5)$.
By symmetry, each vertex in $S(1,2,3)\cup S(2,3,4)$ is either complete or anti-complete to $S(2,3)$.}\label{item:s45}

Suppose first that $s\in S(3,4,5)$ is adjacent to a vertex $u\in S(4,5)$ and
non-adjacent to a vertex $v\in S(4,5)$. By \ref{item:s2clique}, $u$ and $v$ are adjacent. Then
$v, u, s, 3, 2, 1$ induces a $P_6$, a contradiction. So, $S(3,4,5)$ is either complete
or anti-complete $S(4,5)$. Note that the above argument does not use $t$. So,
$S(4,5,1)$ is either complete or anti-complete $S(4,5)$ by symmetry. \e

\emitem {$S(4,5)$ is anti-complete to $S(2,3,4)$. By symmetry, $S(2,3)$ is anti-complete to $S(3,4,5)$. }\label{item:s234s45}

Suppose that $s\in S(2,3,4)$ is adjacent to a vertex $x'\in S(4,5)$.
By \ref{item:s3}, $s$ and $t$ are not adjacent.
Then $\{s,x',t,2\}$ induces a $C_4$, a contradiction. 
\e

\emitem {$S(1,2,3)$ is complete to $S(5,1,2)$. By symmetry, $S(4,5,1)$ is complete to $S(5,1,2)$. }\label{item:s123s512}

Let $s\in S(1,2,3)$ and $t'\in S(5,1,2)$ be two arbitrary vertices. 
Note that $s$ and $x$ are not adjacent by \ref{item:s3antics2}.
By \ref{item:s512}, $t'$ is either complete or anti-complete to $S_2$.
If $t'$ is complete to $S_2$, then $s$ is adjacent to $t'$, for otherwise
$1,t',x,4,3,s,1$ induces a $C_6$ which contradicts that $G$ is $C_6$-free. 
So, we assume that $t'$ is anti-complete to $S_2$.  Suppose that $s$ is not adjacent to $t'$.
Then $H'=H\setminus \{1\}\cup \{t'\}$ is isomorphic to $F_2$ and its underlying five-cycle
is $C'=C\setminus \{1\}\cup \{t'\}$. 
Since $t'$ is anti-complete to $S_2\cup \{s\}$, the number of $2$-vertices with respect
to $C'$ is more than that with respect to $C$.
This contradicts the maximality of $H$.
\e

\emitem {Let $y'$ be an arbitrary vertex in $S(2,3)$. Then $y'$ is complete to $S(1,2,3)$,
and $N(y')\cap S(2,3,4)$ is complete to $S(1,2,3)$.
By symmetry,  any vertex in $S(4,5)$ is complete to $S(4,5,1)$ and the neighbors of that vertex
in $S(3,4,5)$ are complete to $S(4,5,1)$.} \label{item:nbr of y}

Let $s$ be an arbitrary vertex in $S(1,2,3)$. Then $s$ is adjacent to $t$ by \ref{item:s123s512}.
Recall that $t$ is complete to $S_2$ and so is adjacent to $y'$. If $y'$ is not adjacent to $s$,
then $\{s,t,y',3\}$ induces a $C_4$. This shows that $y'$ is complete to $S(1,2,3)$.
Now suppose that $y'$ has a neighbor $r\in S(2,3,4)$ that is not adjacent to $s\in S(1,2,3)$.
Then $s,1,5,4,r,y',s$ induces a $C_6$, a contradiction.
\e

\emitem {$S_5$ is complete to $S_2$.}\label{item:s5s2}

By symmetry, it suffices to show that $S_5$ is complete to $S(2,3)$. Let $u\in S_5$
and $y'\in S(2,3)$ be two arbitrary vertices. Note that $u$ and $t$ are adjacent by \ref{item:s5},
and that and $y' $and $t$ are adjacent by \ref{item:s512}.
If $u$ and $y'$ are not adjacent, then $\{u,t,y',3\}$ induces a $C_4$, a contradiction.
This proves the claim. \e
\end{enumerate}

Now we show that one of $1$, $x$ and $y$ is small. Let  $Q_i=S(i-1,i,i+1)\cup \{i\}$.
Note that $N_G(1) = K\cup Q_2\cup Q_5$, where $K=S_5\cup S(5,1,2)$.
Recall that $S_5$ is complete to $S_3$ by \ref{item:s5}.
It follows from \ref{item:s123s512} that $K\cup Q_2$ and $K\cup Q_5$ are cliques.
If one of $Q_2$ and $Q_5$ has size at most $\omega(G)/2$, then $d_G(1)\le \frac{3}{2}\omega(G)-1$.
So, we assume that $|Q_2|>\omega(G)/2$ and $|Q_5|>\omega(G)/2$.
This implies that $|S(5,1,2)|<\omega(G)/2$. Recall that  $S_1=S(i)$ for some $i\in \{2,5\}$.
By symmetry, we may assume that $S(2)=\emptyset$.
Note that $S(2,3)$ is anti-complete to $S(4,5,1) $ by \ref{item:s3antics2}.
So, by \ref{item:s234s45}, we have that $N_G(y)=K'\cup (N(y)\cap S(5,1,2))$
where $K'=Q_2\cup (N(y)\cap Q_3)\cup (S(2,3)\setminus \{y\})\cup S_5$.
By \ref{item:s45}, \ref{item:nbr of y} and \ref{item:s5s2}, $K'\cup \{y\}$ is a clique
and so $|K'|\le \omega(G)-1$. Therefore, $d_G(y)\le (\omega(G)-1)+|N(y)\cap S(5,1,2)|\le (\omega(G)-1)+|S(5,1,2)|
\le \frac{3}{2}\omega(G)-1$. This completes our proof.
\end{proof}

\begin{lemma}\label{lem:C5}
If a $(P_6,C_4,C_6,F_2)$-free atom $G$ contains a $C_5$, then $G$ contains a small vertex or a universal vertex.
\end{lemma}

\begin{proof}
Let $C=1,2,3,4,5,1$ be an induced $C_5$ of $G$. 
We partition $V(G)\setminus C$
with respect to $C$.  
We choose $C$ such that $|S_3|$ is minimized.
We first prove the following claim which makes use of the choice of $C$.

\begin{enumerate}[label=\bfseries (\arabic*)]

\emitem {For each $1\le i\le 5$, if $S(i-2,i-1)\cup S(i+1,i+2)=\emptyset$, 
then $S(i-1,i,i+1)$ is complete to $S(i-2,i-1,i)\cup S(i,i+1,i+2)$.}\label{clm:s3}

By symmetry, it suffices to prove the claim for $i=1$.
Suppose now that $S(5,1,2)$ is not complete to $S(4,5,1)\cup S(1,2,3)$.
Then there exist vertices $s\in S(5,1,2)$ and $t\in S(4,5,1)\cup S(1,2,3)$
that are not adjacent. Consider the induced five-cycle $C'=C\setminus \{1\}\cup \{s\}$.
Note that $t$ is not a $3$-vertex with respect to $C'$. By the choice of $C$,
there must exist a vertex $r\in V(G)$ that is a $3$-vertex for $C'$ 
but not for $C$. It is routine to check that such a vertex $r$ must necessarily 
lie in $S(2,3)\cup S(4,5)$. However, this contradicts our assumption that
$S(2,3)\cup S(4,5)=\emptyset$. \e

\emitem {For any $1\le i\le 5$, if $S(i)\neq \emptyset$ and it is anti-complete to $S_2$, then $G$ contains an induced $F_2$.}\label{clm:s1}

By symmetry, it suffices to show the claim for $i=1$.
Let $u\in S(1)$. 
By \ref{item:s1}, $S(1)$ is anti-complete to $S(2)\cup S(5)$. If $u_1\in S(1)$
is adjacent to $u_3\in S(3)$, then $u_1$ is adjacent to $u_3$ by \ref{item:s1} and so $1, u_1, u_3, 4, 5, 1$ induces a $C_6$. 
So, $S(1)$ is anti-complete to $S(3)$. By symmetry, $S(1)$ is anti-complete to $S(4)$.
Moreover, $S(1)$ is anti-complete to $S(2,3,4)\cup S(3,4,5)$ by \ref{item:s1antics3}.
 Thus, the neighbors of vertices in $S(1)$
are in $S(4,5,1)\cup S(5,1,2)\cup S(1,2,3)$ and $S_5$.
\\
\begin{enumerate}[label=\bfseries (\alph*)]
\emitem {$S(1)$ is complete to $S(5,1,2)$.}\label{itm:s1s512}

This follows directly from \autoref{lem:P6C4atom}. \e

\emitem {Every vertex in $S(4,5,1)\cup S(1,2,3)$ is either complete or anti-complete to each component of $S(1)$.}\label{s451s123}

By symmetry, we prove this for $S(4,5,1)$. Suppose that some vertex $t\in S(4,5,1)$ is neither complete nor anti-complete
to a component $A$ of $S(1)$. Then by the connectivity of $A$, there exists an edge $aa'$ in $A$ such that $t$ is
adjacent to $a$ but not to $a'$. Then $a',a,t,4,3,2$ induces a $P_6$. \e

\emitem {Let $A$ be a component of $S(1)$. Then the set of neighbors of $A$ in $S(4,5,1)$ is complete to $S(5,1,2)$.
By symmetry, the set of neighbors of $A$ in $S(1,2,3)$ is complete to $S(5,1,2)$. }\label{itm:s451s512}

Let $s\in S(4,5,1)$ be a neighbor of $A$, and $r\in S(5,1,2)$ be an arbitrary vertex.
Let $a$ be an arbitrary vertex in $A$. 
Note that $a$ and $r$ are adjacent by \ref{itm:s1s512}.
If $s$ and $r$ are not adjacent, then $\{5,r,a,s\}$ induces a $C_4$  \e 
\end{enumerate}

\bigskip
Let $A$ be the component of $S(1)$ containing $u$,
$X\subseteq S(4,5,1)$ be the set of neighbors of $A$
in $S(4,5,1)$, and $Y\subseteq S(1,2,3)$ be the set of neighbors of $A$
in $S(1,2,3)$. Note that both $X$ and  $Y$ are cliques by \ref{item:s5}.
It follows from \ref{itm:s451s512} that $N(A)\setminus X$
and $N(A)\setminus  Y$ are cliques. Since $G$ has no clique cutsets,
both $X$ and $Y$ are not empty, say $x\in X$ and $y\in  Y$.
By \ref{s451s123}, $u$ is adjacent to $x$ and $y$.
Now $C\cup \{u,x,y\}$ induces a $F_2$.  \e

\emitem {If $S(i,i+1)\neq \emptyset$ and is anti-complete to $S_1$, then $G$ contains a small vertex.}\label{clm:s2}

By symmetry, assume that $S(3,4)\neq \emptyset$.
It follows from \ref{item:s2} and the $C_6$-freeness of $G$ that
$S_2=S(3,4)\cup S(i,i+1)$ for some $i\in \{5,1\}$, and $S(3,4)$ is anti-complete  $S_2\setminus S(3,4)$.
\\
\begin{enumerate}[label=\bfseries (\alph*)]
\emitem {Any vertex in $S(3,4)$ is anti-complete to either $S(4,5,1)$ or $S(1,2,3)$.}\label{itm:s34}

Suppose not. Let $x\in S(3,4)$ be a vertex that has a neighbor $s\in S(1,2,3)$ and $t\in S(4,5,1)$.
Note that $s$ and $t$ are not adjacent by \ref{item:s3}.
Then $\{1,s,x,t\}$ induces a $C_4$, a contradiction. \e

\emitem{Each vertex in $S(3,4)$ has a neighbor in $S(4,5,1)\cup S(1,2,3)$.}\label{itm:s34nbr}

Let $Y\subseteq S(3,4)$ be the set of vertices that have a neighbor in $S(4,5,1)\cup S(1,2,3)$,
and $S'(3,4)=S(3,4)\setminus  Y$. Suppose that \ref{itm:s34nbr} does not hold, namely $S'(3,4)\neq \emptyset$.
We shall show that $G$ contains a clique cutset, which is a contradiction.
Let $b\in S'(3,4)$ and $B$ be the component of $S'(3,4)$ containing $b$. 
\\
\\
Let $x\in  S(2,3,4)\cup S(3,4,5)\cup Y$ be a vertex that is adjacent to a vertex
$b'\in B$ and not adjacent to a vertex $b''\in B$ such that $b'$ and $b''$ are
adjacent. If $x\in S(2,3,4)$, then $b'', b', x, 2, 1, 5$ induces a $P_6$; 
if $x\in S(3,4,5)$, then $b'', b', x, 5, 1, 2$ induces a $P_6$. If $x\in Y$ , let $t\in S(4,5,1)\cup S(1,2,3)$ be
a neighbor of $x$. If $t\in S(4,5,1)$, then $b'',b',x,t,1,2$ induces a $P_6$, and if
$t\in S(1,2,3)$, then $b'',b',x,t,1,5$ induces a $P_6$. This shows that any vertex
in $S(2,3,4)\cup S(3,4,5)\cup Y$  is either complete or anti-complete to $B$.
Let $Y'$, $S'(2,3,4)$ and $S'(3,4,5)$ be the subsets of $Y$, $S(2,3,4)$
and $S(3,4,5)$, respectively that are complete to $B$.
Let $V'=\{3,4\}\cup S'(2,3,4)\cup S'(3,4,5)$.
If $p\in S'(2,3,4)$ is not adjacent to $q\in S'(3,4,5)$, then $p,b,q,5,1,2,p$ induces a $C_6$,
a contradiction.  Recall that $S(i-1,i,i+1)$ is a clique by \ref{item:s5}. This shows that $V'$ is a clique.
\\
\\
Now let $y_1,y_2\in  Y'$ be two arbitrary and distinct vertices. Suppose that
$y_1$ and $y_2$ are not adjacent. Let $x_i\in S(4,5,1)\cup S(1,2,3)$ be a neighbor of $y_i$
for $i=1,2$. If $x_1=x_2$, then $\{b,y_1,x_1,y_2\}$ induces a $C_4$.
So, $x_1\neq x_2$ and this implies that $x_1$ (resp. $x_2$) is not adjacent to $y_2$
(resp. $y_1$).
We may assume by symmetry that $x\in S(4,5,1)$. Then $y_1, b, y_2, x_2, 1, 2$ induces a $P_6$.
So, $Y'$ is a clique.
Suppose now that $y\in Y'$ is not adjacent to $s\in S(3,4,5)$.
Let $x\in S(4,5,1)\cup S(1,2,3)$ be a neighbor of $y$.
Since $\{s,b,y,x\}$ does not induce a $C_4$,  it follows that $s$ and $x$ are not adjacent. 
If $x\in S(4,5,1)$, then $s,b,y,x,1,2$ induces a $P_6$; if $x\in S(1,2,3)$,
then $s,b,y,x,1,5$ induces a $C_6$. Therefore, $Y'$ is complete to $V'$.
So, we have proved that $Y'\cup V'$ is a clique. Clearly, $S_5\cup V'$ is a clique.
Note that $B$ is anti-complete to $S(5,1,2)$ by \ref{item:s3antics2}, 
and  is anti-complete to $S_1$ by our assumption. Thus, all possible neighbors of
$B$ are in $S_5\cup V'\cup Y'$.
Since $G$ has no clique cutsets, $B$ must have a pair $\{u,y\}$ of neighbors that are not adjacent,
where $u\in S_5$ and $y\in Y'$. But now $\{b,y,x,u\}$ induces a $C_4$,
where $x\in S(4,5,1)\cup S(1,2,3)$ is a neighbor of $y$.
This completes the proof.
\e

\end{enumerate}

\bigskip
By \ref{itm:s34} and \ref{itm:s34nbr}, we can partition $S(3,4)$ into 2 subsets:
$$X_2=\{y\in S(3,4): y \text{ has a neighbor in } S(1,2,3)\},$$
$$X_5=\{y\in S(3,4): y \text{ has a neighbor in } S(4,5,1)\}.$$
Note that $X_2$ is anti-complete to $S(4,5,1)$ and $X_5$ is anti-complete to $S(1,2,3)$.
\\
\begin{enumerate}[label=\bfseries (\alph*)]
\setcounter{enumii}{2}

\emitem {$X_2$ is anti-complete to $X_5$.}\label{itm:x2x5}

Suppose that $y\in X_2$ and $z\in X_5$ are adjacent. 
Let $t_2\in S(1,2,3)$ and $t_5\in S(4,5,1)$ be neighbors of $y$ and $z$, respectively.
Then $5,t_5,z,y,t_2,2$ induces a $P_6$, a contradiction. \e

\emitem {If both $X_2$ and $X_5$ are not empty, then $G$ contains an induced $F_2$.}\label{itm:empty}

Let $x\in X_5$ and $y\in X_2$. By \ref{itm:s34nbr}, $x$ has a neighbor $t\in S(4,5,1)$
and $y$ has a neighbor $s\in S(1,2,3)$. Note also that $x$ and $y$ are not adjacent by \ref{itm:x2x5}.
Now $(C\setminus \{5\})\cup \{x,y,s,t\}$ induces a $F_2$ (whose underlying five-cycle is $y,s,1,t,4,y$).  \e

\end{enumerate}

By \ref{itm:empty} the fact that $G$ is $F_2$-free, we may assume, without loss of generality, that $X_5=\emptyset$.
In other words, every vertex in $S(3,4)$ has a neighbor in $S(1,2,3)$. 
Let $x'\in S(3,4)$ and $s'\in S(1,2,3)$ be a neighbor of $x'$.
If $S(1,5)$ contains a vertex $y$, then either $2,s',x',4,5,y$ induces a $P_6$
or $C\cup \{x',y,s'\}$ induces a subgraph isomorphic to $F_2$, depending on whether $s'$ and $y$
are adjacent. This shows that $S(1,5)=\emptyset$.
So, $S_2=S(3,4)\cup S(1,2)$.
\\
\begin{enumerate}[label=\bfseries (\alph*)]
\setcounter{enumii}{4}

\emitem {$S(3,4)$ is complete to $S_5$.}\label{itm:s5s34}

Suppose not. We may assume that $x\in S(3,4)$ is not adjacent to some vertex
$u\in S_5$. Let $t\in S(1,2,3)$ be a neighbor of $x$.
Then $\{u,t,x,4\}$ induces a $C_4$, a contradiction. \e

\emitem {Each vertex in $S(2,3,4)$ is either complete or anti-complete to any component of $S(3,4)$.}\label{itm:s234s345}  

Suppose that $t\in S(2,3,4)$ is neither complete nor anti-complete some component $A$ of $S(3,4)$.
Then by the connectivity of $A$, there exists an edge $yz$ in $A$ such that $t$ is adjacent to $y$ but not to $z$.
Now $z,y,t,2,1,5$ induces a $P_6$, a contradiction. \e

\emitem {$S(3,4)$ is  complete to $S(3,4,5)$.}\label{itm:s345s34}

Suppose that $x\in S(3,4)$ and $t\in S(3,4,5)$ are not adjacent.
Let $s\in S(1,2,3)$ be a neighbor of $x$. 
Note that $s$ is not adjacent to $t$ by \ref{item:s3}.
Then $C\cup \{x,s,t\}$
induces a subgraph isomorphic to $F_2$ whose underlying five-cycle is $x,s,1,5,4,x$. This is a contradiction. \e

\emitem {$S(3,4,5)$ is complete to $S(2,3,4)$ and $S(4,5,1)$, and $S(1,2,3)$ is complete to $S(5,1,2)$.}\label{itm:s3complete}

Recall that $S(1,5)\cup S(2,3)=\emptyset$ and $S(2,3)\cup S(4,5)=\emptyset$. Thus, \ref{itm:s3complete}
follows from \ref{clm:s3}.
\e

\emitem {$S(3,4)$ is $P_4$-free.}\label{itm:s34p4free}

Suppose that $P=x_1,x_2,x_3,x_4$ is an induced $P_4$ in $S(3,4)$.
Let $s\in S(1,2,3)$ be a neighbor of $x_4$. 
Note that $s$ is adjacent to neither $x_1$ nor $x_2$, for otherwise $\{s,x_i,4,x_4\}$ induces a $C_4$
for some $i\in \{1,2\}$.  Now $P\cup \{s,1,5\}$ contains a $P_6$. \e
\end{enumerate}

\bigskip
By \ref{itm:s34p4free} and the fact that $G$ is $C_4$-free, $S(3,4)$ is chordal.
So, $S(3,4)$ contains a vertex $x$ such that $X=N(x)\cap S(3,4)$ is a clique.
Now we show that either $x$ or $5$ is small. For each $i$, let $Q_i=S(i-1,i,i+1)\cup \{i\}$
and $S'(i-1,i,i+1)=N(x)\cap S(i-1,i,i+1)$. 
Note that $S(3,4)$ is anti-complete to $S(5,1,2)$ by \ref{item:s3antics2}, and is anti-complete to $S_1$ by our assumption.
Moreover, recall that $S_2=S(3,4)\cup S(1,2)$, and that $S(3,4)$ is anti-complete to $S(1,2)$ by \ref{item:s2}.
It follows from \ref{itm:s5s34} and \ref{itm:s345s34} that
$N_G(x)=S_5\cup S(3,4,5)\cup S'(2,3,4)\cup X\cup S'(1,2,3)$.
It follows from \ref{itm:s5s34}-\ref{itm:s3complete} that $N_G(x)\setminus S'(1,2,3)$ is a clique.
If $|Q_2|\le \omega(G)/2$, then 
$d_G(x)= |N_G(x)\setminus S'(1,2,3)|+|S'(1,2,3)|\le \omega(G)-1+|Q_2|\le \frac{3}{2}\omega(G)-1$.
Otherwise, $|Q_2|>\omega(G)/2$. 
By \ref{itm:s3complete}, $|Q_1|<\omega(G)/2$. Note that $S(5)=\emptyset$ by \ref{item:S1S2empty}.
Then $d_G(5)=|S(4,5,1)\cup Q_4\cup S_5|+|Q_1|< \omega(G)-1+\omega(G)/2=\frac{3}{2}\omega(G)-1$,
since $S(4,5,1)\cup Q_4\cup S_5$ is a clique by \ref{itm:s3complete}.
This completes the proof of \ref{clm:s2}. \e

\end{enumerate}

We now complete the proof of the lemma as follows.
Suppose first that $S_2\neq \emptyset$, say $S(3,4)\neq \emptyset$.
Let $x\in S(3,4)$. By \ref{clm:s2}, $x$ has a neighbor 
$u\in S_1$. By \ref{item:s1s2}, $u\in S(1)$.
Thus, $S_1=S(1)$, and $S_2=S(3,4)\cup S(i,i+1)$ for some $i\in \{5,1\}$.
If $S(5,1)$ or $S(1,2)$ is not empty, then by \ref{clm:s2} the fact that $S(1)$ is anti-complete to $S(1,2)\cup S(5,1)$
it follows that $G$ contains a small vertex.
So, $S_2=S(3,4)$. Note that $N_G(i)\subseteq S_3\cup S_5$ for $i=2,5$ since $S(3,4)\neq \emptyset$.
Furthermore, it follows from \ref{clm:s3} that $S(i-1,i,i+1)$
is complete to $S(i,i+1,i+2)$ for each $i$. 
If $S(3,4,5)\cup \{4\}\le \omega(G)/2$, then $d_G(5)\le \frac{3}{2}\omega(G)-1$;
otherwise $S(2,3,4)\cup \{3\}<\omega(G)/2$ and thus $d_G(2)\le \frac{3}{2}\omega(G)-1$.
 Therefore, either $2$ or $5$ is a small vertex.
 
Therefore, we assume that $S_2=\emptyset$. If $S(i)\neq \emptyset$ for some $i$, then $G$ contains an induced 
$F_2$ by \ref{clm:s1} and this contradicts our assumption that $G$ is $F_2$-free.
So, $S_1=\emptyset$. If $S_5\neq \emptyset$, then $G$ contains a universal vertex by \ref{item:s5} and we are done.
Otherwise $G$ is a blow-up of $C$ and so contains a small vertex.
This completes our proof of the lemma.
\end{proof}

We are now ready to prove \autoref{thm:main}.

\begin{proof}[of \autoref{thm:main}]
 Let $G$ be a $(P_6,C_4)$-free atom.
It follows from \autoref{lem:F1}--\autoref{lem:C5} that we can assume that $G$
is also $(C_6,C_5)$-free. Therefore, $G$ is chordal. It is well-known \cite{Di61} that every chordal
graph contains a vertex of degree at most $\omega(G)-1$ and so this vertex is small. 
This completes the proof.
\end{proof}

\section{$\chi$-Bounding $(P_6,C_4)$-free graphs}\label{sec:3/2}

In this section, we shall prove the main result of this paper, that is, 
every  $(P_6,C_4)$-free graph has $\chi\le \frac{3}{2}\omega$.
For that purpose, we need one additional lemma.

\begin{lemma}\label{lem:blowup}
Let $G$ be a graph
and let $H$ be the skeleton of $G$. If $\chi(H)\le 3$,
then $\chi(G)\le \frac{3}{2}\omega(G)$.
\end{lemma}

\begin{proof}
We prove this by induction on $|V(G)|$.
The base case is that $G$ is its own skeleton.
Our assumption implies that $\chi(G)\le 3$. If $\omega(G)\ge 2$, then
it follows that $\chi(G)\le \frac{3}{2}\omega(G)$.
Otherwise, $\omega(G)=1$, i.e., $G$ is an independent set.
So, $\chi(G)=1<\frac{3}{2}\omega(G)$.
Now suppose that the lemma is true for any graph $G'$ with
$|V(G')|<|V(G)|$. 
If $G$ is disconnected, then the lemma follows from applying the inductive hypothesis to each connected component of $G$.
So, $G$ is connected and thus $H$ is also connected.
Then any vertex of $H$ lies in a maximal clique of $H$ with size at least 2. 
Note that every maximal clique of $G$ consists of a maximal clique of $H$ with all its twins.
This implies that $\omega(G')\le \omega(G)-2$,
where $G'=G-H$.
Note that the skeleton $H'$ of $G'$ is an induced subgraph of $H$
and so $\chi(H')\le 3$.
By the inductive hypothesis it follows that $\chi(G')\le \frac{3}{2}\omega(G')$.
Finally,
\[\chi(G)\le \chi(G')+\chi(H)\le \frac{3}{2}\omega(G')+\chi(H)\le \frac{3}{2}(\omega(G)-2)+3=\frac{3}{2}\omega(G).\]
This completes our proof.
\end{proof}

Now we are ready to prove the main result of this paper.
\begin{theorem}\label{thm:main2}
Every $(P_6,C_4)$-free graph $G$ has $\chi(G)\le \frac{3}{2}\omega(G)$.
\end{theorem}

\begin{proof}
We use induction on $|V(G)|$. We may assume that $G$ is connected, for otherwise
we apply the inductive hypothesis on each connected component.
If $G$ contains a clique cutset $K$ that disconnects $H_1$ from $H_2$,
let $G_i=G[H_i\cup K]$ for $i=1,2$. Then the inductive hypothesis implies that
$\chi(G_i)\le \frac{3}{2}\omega(G_i)$ for $i=1,2$. Note that $\chi(G)=\max\{\chi(G_1),\chi(G_2)\}$
and so $\chi(G)\le \frac{3}{2}\omega(G)$. Now $G$ is an atom.
If $G$ contains a universal vertex $u$, then applying the inductive hypothesis to $G-u$
implies that $\chi(G-u)\le \frac{3}{2}\omega(G-u)$. Since $\chi(G)=\chi(G-u)+1$
and $\omega(G)=\omega(G-u)+1$, it follows that $\chi(G)\le \frac{3}{2}(\omega(G)-1)+1< \frac{3}{2}\omega(G)$.
 If $G$ contains a small vertex $v$, then applying
the inductive hypothesis to $G-v$ gives us that $\chi(G-v)\le \frac{3}{2}\omega(G-v)\le \frac{3}{2}\omega(G)$.
Since $v$ has degree at most $\frac{3}{2}\omega(G)-1$, it follows that $\chi(G)=\chi(G')\le \frac{3}{2}\omega(G)$.
It follows then from \autoref{thm:main} that $G$ is  a blow-up of the Petersen graph or $F$.
In other words, the skeleton of $G$ is the Petersen graph or $F$. 
It is straightforward to check that both graphs have chromatic number $3$. Therefore, $\chi(G)\le \frac{3}{2}\omega(G)$
by \autoref{lem:blowup}. This completes our proof.
\end{proof}

\section{A 3/2-Approximation Algorithm}\label{sec:alg}

In this section, we give a polynomial time 3/2-approximation algorithm for coloring $(P_6,C_4)$-free graphs.
The general idea is to decompose the input graph $G$ in the following way to obtain a decomposition tree $T(G)$
where the leaves of $T(G)$ are `basic' graphs which we know how to color and the internal nodes
of $T(G)$ are subgraphs of $G$ that are decomposed via either clique cutsets or small vertices. 
Since both clique cutsets and small vertices `preserve' the colorability of graphs, a bottom-up approach on $T(G)$ will give
us a coloring of $G$ using at most $\frac{3}{2}\omega(G)$ colors.

Formally, let $G$ be a connected $(P_6,C_4)$-free graph.
If $G$ has a clique cutset $K$, then $G-K$ is a disjoint union of two subgraphs $H_1$ and $H_2$ of $G$.
We let $G_i=H_i\cup K$ for $i=1,2$ and decompose $G$ into $G_1$ and $G_2$.
On the other hand, if $G$ does not contain any clique cutset but contains a small vertex $v$
that is not universal, then we decompose $G$ into $G-v$. 
We then further decompose $G_1$ and $G_2$ or $G-v$
in the same way until either the graph has no clique cutsets and no small vertices  or the graph is a clique. 
We refer to these subgraphs that are not further decomposed as {\em strong atoms}.
The decomposition procedure can be represented by a binary tree $T(G)$ whose root is $G$,
and $G$ may have two children $G_1$ and $G_2$ or only one child $G-v$, depending
on the way $G$ is decomposed. Each leaf in $T(G)$ corresponds to a strong atom.
Let $A$ be a strong atom which is not a clique. Note that if $A$ contains a universal vertex $u$,
then $A-u$ is still a strong atom which is not a clique. 
By repeatedly applying this observation and \autoref{thm:main}, it follows that
any strong atom which is not a clique must be the join of a blow-up of the Petersen graph or $F$ and a clique.
It turns out that there are only polynomially many nodes in $T(G)$ and $T(G)$
can be found in polynomial time, see \autoref{lem:T(G)} and \autoref{lem:buildT(G)} below.

\begin{lemma}\label{lem:T(G)}
$T(G)$ has $O(n^2)$ nodes.
\end{lemma}

\begin{proof}
To see this, it is enough to prove that there are $O(n^2)$ internal nodes in $T(G)$.
We label each internal node $X$ of $T(G)$ with a pair of ordered vertices as follows.
\begin{enumerate}
\item [$\bullet$] If $X$ is decomposed via a clique cutset $K$ into two subgraphs $X_1$ and $X_2$,
then we choose a vertex $a\in X_1-K$ and a vertex $b\in X_2-K$, and label $X$ with $(a,b)$.
\item [$\bullet$] If $X$ is decomposed via removing a small vertex $v$ that is not universal in $X$, 
i.e., $d_X(v)\le \frac{3}{2}\omega(X)-1$,  then we choose a vertex $u$ that is not adjacent to $v$ in $X$ and label $X$ with $(v,u)$. Note that
the choice of $u$ is always possible since $v$ is not universal.
\end{enumerate}

Due to our choice of labeling, if $X$ is labeled with $(x,y)$, then
$x,y\in X$ and $xy\notin E$.
Now we show that no two internal nodes have the same label. Suppose not,
let $A$ and $B$ be two internal nodes  of $T(G)$ that have the same label,
say $(x,y)$. Suppose first that $B$ is a descendant of $A$.
If $A$ is decomposed via a clique cutset into $A_1$ and $A_2$, then
the unique path connecting $A$ and $B$ in the subtree rooted at $A$
goes through either $A_1$ or $A_2$, say $A_1$. Then the fact
that $A$ has label $(x,y)$ implies that $y\notin A_1$.
On the other hand, the fact
that $B$ has label $(x,y)$ implies that $y\in B$.
But this is a contradiction, since $B$ is an induced subgraph of $A_1$.
So, $A$ is decomposed via removing a small vertex.
This means that $A-x$ is the only child of $A$ in $T(G)$ and so
the unique path connecting $A$ and $B$ in the subtree rooted at $A$
goes through $A-x$. Again, the label of $B$ implies that $x\in B$
but this is a contradiction, since $B\subseteq A-x$.

So, we assume that $B$ is not a descendant of $A$. Similarly,
$A$ is not a descendant of $B$.  Let $X$ be the lowest common ancestor
of $A$ and $B$. Then $X$ must be decomposed via a clique cutset $K$ into
two subgraphs $X_1$ and $X_2$,
for otherwise $X-v$ for some $v\in X$ with $d_X(v)<\frac{3}{2}\omega(G)-1$ would have been
a common ancestor that is lower than $X$.  For the same reason, $A$ and $B$
lie in the subtree rooted at $X_1$ and $X_2$, respectively.
As we observed earlier, $x,y\in A\subseteq X_1$ and $x,y\in B\subseteq X_2$.
This implies that $x,y\in X_1\cap X_2=K$. Since $K$ is a clique, $xy\in E$ but
this is a contradiction.

Since there are at most $n^2$ distinct pairs of vertices, the number of internal nodes
is $O(n^2)$.
\end{proof}

\begin{lemma}\label{lem:buildT(G)}
For any graph $G$, $T(G)$ can be found in $O(n^2m)$ time.
\end{lemma}
\begin{proof}
Given a graph $G$, we explain how to decide in $O(m+n)$ time whehter $G$ can be decomposed.
Since there are $O(n^2)$ nodes in $T(G)$ by \autoref{lem:T(G)}, our lemma follows.
First, we use Tarjan's algorithm \cite{Ta85} to find a clique cutset. The algorithm
either returns a clique cutset if one exists or reports that none exist in $O(m)$ time.
If $G$ does contain some clique cutset, then we decompose $G$ into two subgraphs.
Now $G$ is an atom. We first test if $G$ is a clique and this can be done in $O(n)$ time.
If $G$ is a clique, then we do not decompose $G$ and so $G$ will be a leaf.  We then
partition $G$ into equivalence classes of true twins. This can be done in $O(m+n)$ time \cite{RT87}.
If the skeleton of $G'$, where $G'$ is obtained from $G$ by removing all universal vertices, 
is isomorphic to the Peterson graph or $F$, which can be tested in constant time, 
then we do not decompose $G$. Otherwise, by \autoref{thm:main}, $G$ must contain a small vertex. 
Therefore, a vertex, say $v$, of minimum degree will be such a vertex. Finding such a vertex takes $O(m)$ time.
As $G$ is not a clique, $v$ is not universal.  We thus decompose $G$ into $G-v$.
The total running time is therefore $O(m+n)$.
\end{proof}

We now present our $3/2$-approximation algorithm for coloring $(P_6,C_4)$-free graphs.
\begin{theorem}\label{thm:alg}
There is an $O(n^2m)$ algorithm to find a coloring of $G$ that uses at most $\frac{3}{2}\omega(G)$ colors.
\end{theorem}

\begin{proof}
The algorithm works as follows:
\begin{inparaenum}[(i)]
\item we first find $T(G)$;
\item color each leaf $X$ of $T(G)$ using at most  $\frac{3}{2}\omega(X)$ colors;
\item for an internal node $Y$, if $Y$ has only one child $Y-y$, then color $y$ with a color that is not used on its neighbors
in $Y$; if $Y$ has two children $Y_1$ and $Y_2$, then combine the colorings of $Y_1$ and $Y_2$ on the clique cutset that decomposes $Y$.
\end{inparaenum} 
The correctness follows from \autoref{lem:blowup} and \autoref{thm:main2}.
It takes $O(mn^2)$ time to find $T(G)$ by \autoref{lem:buildT(G)}. 
Moreover, it is easy to see that one can color a leaf $X$ of $T(G)$
in time $O(n+m)$ time, and combining the coloring for a single decomposition step takes $O(n)$ time.
Therefore, it takes $(O(m+n)+O(n))O(n^2)=O(n^2m)$ time to obtain a desired coloring of $G$.
\end{proof}

\noindent {\bf Acknowledgments.}
Serge Gaspers is the recipient of an Australian Research Council (ARC) Future Fellowship (FT140100048)
and acknowledges support under the ARC's Discovery Projects funding scheme (DP150101134).
Shenwei Huang is supported by the National Natural Science Foundation of China (11801284).

We thank one reviewer for his/her careful reading of the paper and many constructive suggestions 
that improve the presentation of the paper greatly.


\end{document}